\documentclass{amsart}

\usepackage[english]{babel}
\usepackage{tikz}
\usetikzlibrary{positioning,arrows}

\usepackage{amsmath,amssymb,amsthm}

\usepackage{enumerate}

\newtheorem{theorem}{Theorem}[section]
\newtheorem{proposition}[theorem]{Proposition}
\newtheorem{lemma}[theorem]{Lemma}
\newtheorem{corollary}[theorem]{Corollary}

\newtheorem{question}[theorem]{Question}		
	
\theoremstyle{definition}
\newtheorem{definition}[theorem]{Definition}

\theoremstyle{remark}
				
\newtheorem{remark}[theorem]{Remark}

\newcommand{\Aut}{\operatorname{Aut}}

\newcommand{\Perm}{\operatorname{Perm}}
\newcommand{\Obj}{\operatorname{Obj}}

\newcommand{\K}{{\mathbb K}}

\newcommand{\G}{\mathcal{G}}
\newcommand{\C}{\mathcal{C}}
\newcommand{\E}{{\mathcal E}}
\newcommand{\X}{{\mathcal X}}

\newcommand{\graphs}{{{\mathcal G}raphs}}

\usepackage[pdfpagelabels]{hyperref}

\AtBeginDocument{%
   \def\MR#1{}
}

\begin{document}

\title{Regular evolution algebras are universally finite}
\author{Cristina Costoya}
\address{CITIC Research Center, Ciencias de la Computaci{\'o}n y Tecnolog{\'\i}as de la Informaci{\'o}n,
Universidade da Coru{\~n}a, 15071-A Coru{\~n}a, Spain.}
\email{cristina.costoya@udc.es}
\thanks{The first author was partially supported by Mi\-nis\-te\-rio de Econom\'ia y Competitividad (Spain), grant  MTM2016-79661-P}
\author{Panagiote Ligouras}
\address{I.I.S.S.\ ``Da Vinci-Agherbino" - Via Repubblica, 36/H, 70015 Noci BA, Italy}
\email{ligouras@alice.it}
\author{Alicia Tocino}
\address{Departamento de \'Algebra, Geometr{\'\i}a y Topolog{\'\i}a, Universidad de M{\'a}laga, 29071-M{\'a}laga, Spain}
\email{alicia.tocino@uma.es}
\thanks{The third author was partially supported by by Ministerio de Econom\'ia y Competitividad (Spain) grant MTM2016-76327-C3-1-P}

\author{Antonio Viruel}
\address{Departamento de \'Algebra, Geometr{\'\i}a y Topolog{\'\i}a, Universidad de M{\'a}laga, 29071-M{\'a}laga, Spain}
\email{viruel@uma.es}
\thanks{The fourth author was partially supported by by Ministerio de Econom\'ia y Competitividad (Spain) grant MTM2016-78647-P}


\begin{abstract}
In this paper we show that evolution algebras over any given field $\Bbbk$ are universally finite. In other words, given any finite group $G$, there exist infinitely many regular evolution algebras $X$ such that $\Aut(X)\cong G$. The proof is built upon the construction of a covariant faithful functor from the category of finite simple (non oriented) graphs to the category of (finite dimensional) regular evolution algebras.
\end{abstract}

\maketitle

\section{Introduction}

Given a category $\C$, we say that $\C$ is finitely universal \cite[Section 4.1]{Babai} if every finite group can be represented as the full automorphism group of an object in $\C$, that is, if given $G$ a finite group, there exists $X\in\Obj(\C)$ such that $\Aut_{\C}(X)\cong G$.

Deciding whether a given category is finitely universal is a hard problem. From the celebrated result of Frucht \cite{Frucht39}, who considered finite graphs, to the recent results in \cite{CV2,CV-advances} (resp.\ \cite{CMV1,cmv3}), where the homotopy category of spaces (resp.\ the category of differential graded algebras) is considered,  mathematicians have been addressing the identification of finitely universal categories for almost a century. The reader is encouraged to consult \cite[Introduction]{Jones} for a (non exhaustive) account of milestones in this question.

In this manuscript we consider the property of being finitely universal for evolution algebras over an arbitrary field $\Bbbk$ (see Section \ref{sec:evolution} for precise definitions). Evolution algebras are, non necessarily associative, commutative algebras that were originally introduced by Tian \cite{Tian} (see also \cite{Tian-Vojtechovsky}) as a mathematical model of self-reproduction in non-Mendelian genetics. Despite of the lack of nice structural properties, for example they are not power associative nor closed under subalgebras, evolution algebras have been a very active research topic the last decade (see \cite{Camacho-Gomez,Tian-Zou,Casas-Ladra,Rozikov-Murodov} and the references provided there).

Within this context, our main result is:

\begin{theorem}\label{thm:main}
Let $\E_\Bbbk$ be the category of evolution algebras over the field $\Bbbk$, and $G$ be a finite group. Then there are infinitely many non isomorphic $X\in\Obj(\E_\Bbbk)$ such that
\begin{enumerate}
\item $X^2=X$, that is, $X$ is regular, and
\item $\Aut_{\E_\Bbbk}(X)\cong G$.
\end{enumerate}
In particular the category $\E_\Bbbk$ is finitely universal.
\end{theorem}
\begin{proof}
Given a finite group $G$, there exist infinitely many connected simple graphs $\G$ such that $\Aut_\graphs(\G)\cong G$ \cite{Frucht39} (see also \cite{deGroot} and \cite{Sabidussi}). Therefore, according to Theorem \ref{thm:funtor}, there are infinitely many regular evolution algebras $X:=\X(\G)\in\Obj(\E_\Bbbk)$ such that $\Aut_{\E_\Bbbk}(X)\cong \Aut_\graphs(\G)\cong G$.
\end{proof}

Since the automorphism group of a regular evolution algebra is always finite, \cite[Theorem 4.8]{Elduque-Labra-2015}, we obtain the following:

\begin{corollary}\label{cor:main}
Let $G$ be an abstract group. Then $G$ is finite if and only if there exists a regular evolution algebra $X$ such that $\Aut_{\E_\Bbbk}(X)\cong G$.
\end{corollary}

If $X$ is a non necessarily regular evolution algebra over a field $\Bbbk$, and $\K$ denotes the algebraic closure of $\Bbbk$,
then the group $G=\Aut_{\E_{\K}}(X\otimes_\Bbbk \K)$ is indeed an algebraic linear group defined over $\Bbbk$ such that $G(\Bbbk)=\Aut_{\E_\Bbbk}(X)$. Therefore, in the spirit of Gordeev-Popov's work \cite{MR2031860}, Theorem \ref{thm:main} may be seen as a strengthening  of \cite[Corollary 2]{MR2031860}, and it is natural to consider the following question (compare to \cite[Theorem 1]{MR2031860}):

\begin{question}\label{qu:algebraic}
Let $G$ be an algebraic group defined over $\Bbbk$. Does there exist an evolution algebra defined over $\Bbbk$, $X$, such that the algebraic group $\Aut_{\E_{\K}}(X\otimes_\Bbbk \K)$ is $\Bbbk$-isomorphic to $G$?
\end{question}

\begin{remark}
Notice that in contrast to \cite[Theorem 1]{MR2031860}, the algebras in Question \ref{qu:algebraic} are not required to be simple since the group $G$ is not assumed to be finite. Indeed, simple evolution algebras are regular \cite[Corollary 4.6]{Cabrera-Siles-Velasco}, and therefore as we have mentioned above, their automorphism groups are finite.
\end{remark}

%
%


\section{Basics on evolution algebras}\label{sec:evolution}

In this section we briefly introduce the basics on evolution algebras that we need in the forthcoming sections. While the standard reference on the subject is \cite{Tian}, our notation is closer to the more recent work \cite{Tesis-Yolanda}.

Given a field $\Bbbk$, we say that $A$ is an algebra over $\Bbbk$, or a $\Bbbk$-algebra, if $A$ is a $\Bbbk$-vector space furnished with an inner product, a $\Bbbk$-bilinear map $A\times A\to A$. An algebra morphism is then a linear map which commutes with inner products.  The algebras considered in this manuscript are all finite dimensional vector spaces unless otherwise is stated.

We now recall the definition of an evolution algebra \cite[Definition 1]{Tian}, \cite[Definitions 1.2.1]{Tesis-Yolanda}:

\begin{definition}\label{def:evolution}
An \textit{evolution algebra} over a field $\Bbbk$ is a $\Bbbk$-algebra $X$ provided with a basis $B = \{b_i\, |\, i \in \Lambda\}$ such that $b_i b_j = 0$ whenever $i\neq j$. Such a basis $B$ is called a \textit{natural basis}.
\end{definition}

Every $\Bbbk$-vector space $X$ gives rise to a degenerated evolution algebra by defining $xy=0$ for every $x,y\in X$. More generally, given a $\Bbbk$-vector space $X$ spanned by a basis $B = \{b_i\, |\, i \in \Lambda\}$, an evolution algebra with natural basis $B$ is completely determined by just giving the linear expression of $b_i^2:=b_i b_i\in X$  for every $i\in\Lambda$. This motivates the following definition \cite[Definitions 1.2.1]{Tesis-Yolanda}:

\begin{definition}\label{def:structure}
Let $X$ be an evolution algebra, and fix a natural basis $B$ in $X$. The scalars $\omega_{ki} \in\Bbbk$ such that
$b_i^2:=b_i b_i=\sum_{k\in\Lambda}\omega_{ki}b_k$ are called the \textit{structure constants} of $X$ relative to $B$, and the matrix $M_B := (\omega_{ki})$ is said to be the \textit{structure matrix} of $X$ relative to $B$.
\end{definition}

Notice that given an evolution algebra $X$, with natural basis $B$, the column vectors of $M_B$ span $X^2$, the algebra generated by twofold products of elements in $X$. Therefore, if $X=X^2$, then $M_B$ must be of maximal rank, that is, $M_B$ must be a regular matrix (see also \cite[Remark 4.2]{Elduque-Labra-2015}). This justifies the following definition:

\begin{definition}\label{def:regular}
We say that an evolution algebra $X$ is \textit{regular} if and only if $X=X^2$, or equivalently, if and only if for any natural basis $B$, the structure matrix $M_B$ is a regular matrix.
\end{definition}

The following result shows that regular evolution algebras have essentially just one natural basis, which plays a key role in our arguments.

\begin{proposition}{\rm(\cite[Theorem 4.4]{Elduque-Labra-2015}, \cite[Proposition 1]{Celorrio-Velasco})}\label{prop:unique_basis}
Let $X$ be a regular evolution algebra over $\Bbbk$ provided with a natural basis $B = \{b_i\, |\, i \in \Lambda\}$. Given any other natural basis $\hat{B} = \{\hat{b}_i\, |\, i \in \Lambda\}$, there exists a permutation $\sigma\in\Perm(\Lambda)$, and scalars $\lambda_i\in\Bbbk$ such that $\hat{b}_i=\lambda_ib_{\sigma(i)}$
\end{proposition}


\section{From simple graphs to evolution algebras}\label{sec:graphs}

In \cite[Definition 15]{Tian} and \cite[Definition 2.2]{Elduque-Labra-2015}, authors have related evolution algebras with weighted directed graphs. In this section we show that simple (non oriented) graphs give rise to evolution algebras in a functorial way.

Let $\graphs$ denote the category of \textit{finite simple graphs}. Objects in $\graphs$ are pairs $\G=(V,E)$ where $V$ is a finite set, the set of vertices of $\G$, and $E$ is a set of unordered pairs of different vertices (subsets of $V$ of cardinality $2$), so we shall frequently denote elements $e\in E$ as $e=\{v,w\}\subset V$. Notice that by definition, two vertices in a simple graph can be connected by at most one edge, and loops are not allowed.

Given two simple graphs $\G_i=(V_i,E_i)$, $i=1,2$, a graph morphism $f\colon \G_1\to \G_2$ is an injective set morphism $f\colon V_1\to V_2$ such that if $e=\{v,w\}\in E_1$, then $f(e)=\{f(v),f(w)\}\in E_2$.

Our main result in this section is the following:

\begin{theorem}\label{thm:funtor}
 Given any field $\Bbbk$, there exists a covariant faithful functor $$\X\colon \graphs \rightsquigarrow \E_\Bbbk$$ such that for every $\G\in\graphs$ the following hold:
\begin{enumerate}
\item $\X(\G)^2=\X(\G)$, that is, $\X(\G)$ is regular, and
\item $\Aut_{\E_\Bbbk}\big(\X(\G)\big)\cong \Aut_\graphs(\G)$.
\end{enumerate}
\end{theorem}

For convenience of the reader, the proof of Theorem \ref{thm:funtor} has been split in several lemmas. Henceforth, $\Bbbk$ is a fixed field.

We first define the functor $\X$ on objects:

\begin{definition}\label{def:funtor_objeto}
Given a simple graph $\G=(V,E)$, we define $\X(\G)$ to be the evolution algebra over $\Bbbk$, with natural basis $B=\{b_v:v\in V\}\cup\{b_e:e\in E\}$ and multiplication given by
\begin{itemize}
\item $b_v^2=b_v$ for all $v\in V$.
\item $b_e^2=b_e+\sum_{v\in e} b_v$ for all $e\in E$.
\end{itemize}
\end{definition}

We now prove that $\X$ actually gives rise to regular evolution algebras.

\begin{lemma}\label{lem:regular}
For any simple graph $\G$, the evolution algebra $\X(\G)$ is regular.
\end{lemma}
\begin{proof}
Let $M_B$ be the structure matrix of $\X(\G)$ relative to the natural basis $B$ given in Definition \ref{def:funtor_objeto}. Then $M_B$ is upper triangular, and every input on the diagonal is equal to $1$. Therefore $\operatorname{det}(M_B)=1$ and $M_B$ is regular, hence so is $\X(\G)$.
\end{proof}

We now define the functor $\X$ on morphisms:

\begin{definition}\label{def:funtor_maps}
Given a graph morphism $f\colon \G_1\to \G_2$ where $\G_i=(V_i,E_i)$, $i=1,2$, define $\X(f)\colon \X(\G_1)\to \X(\G_2)$ to be the algebra morphism given by $\X(f)(b_v)=b_{f(v)}$ and $\X(f)(b_e)=b_{f(e)}$.
\end{definition}

\begin{remark}
The map $\X(f)$ above is a well defined algebra morphism. Indeed, if $f$ is a graph morphism, $f$ maps vertices to vertices, and edges to edges while preserving adjacency. Therefore for any  vertex $v\in V_1$, $f(v)\in V_2$ and
$$\X(f)(b_v)^2=b_{f(v)}^2=b_{f(v)}=\X(f)(b_v),$$
while for any edge $e=\{v,w\}\in E_1$, $f(e)=\{f(v),f(w)\}\in E_2$ and
$$\X(f)(b_e)^2=b_{f(e)}^2=b_{f(e)}+b_{f(v)}+b_{f(w)}=\X(f)(b_{e}+b_{v}+b_{w})=\X(f)(b_e^2).$$
Finally, it is straightforward to check that $\X(f)=\X(g)$ if and only if $f=g$, and $\X(f)\circ \X(g)=\X(f\circ g)$.
\end{remark}

The functor $\X$ is also faithful on objects.

\begin{lemma}\label{lem:fiel-objetos}
Let $\G_i=(V_i,E_i)$, $i=1,2$, be simple graphs. Then $\X(\G_1)\cong \X(\G_2)$ if and only if $\G_1\cong \G_2$.
\end{lemma}
\begin{proof}
If $\G_1\cong \G_2$, functoriality provides the algebra isomorphism $\X(\G_1)\cong \X(\G_2)$.

Assume now that $\psi\colon \X(\G_1)\cong \X(\G_2)$ is an isomorphism of regular evolution algebras. Let $B_i$ be the natural basis of $\X(\G_i)$, $i=1,2$, given in Definition \ref{def:funtor_objeto}. Then $\psi(B_1)$ is a natural basis of $\X(\G_2)$, and according to Proposition \ref{prop:unique_basis}, there exist scalars $\lambda_v,\lambda_e\in\Bbbk$, for $v\in V_1$ and $e\in E_1$, such that $\lambda_v\psi(b_v),\lambda_e\psi(b_e)\in B_2$. Moreover, the number of non trivial structure constants must be preserved for each element in the basis, hence for $v\in V_1$ and $e\in E_1$ there exist $\sigma(v)\in V_2$ and $\tau(e)\in E_2$ such that $\lambda_v\psi(b_v)=b_{\sigma(v)}$ and $\lambda_e\psi(b_e)=b_{\tau(e)}$.

Comparing
\begin{align*}
b_{\sigma(v)}^2&=b_{\sigma(v)},\text{ and}\\
\big(\lambda_v\psi(b_v)\big)^2&=\lambda_v^2\psi(b_v)^2=\lambda_v^2\psi(b_v^2)=\lambda_v^2\psi(b_v)=\lambda_v\big(\lambda_v\psi(b_v)\big)\\
&=\lambda_vb_{\sigma(v)},
\end{align*}
we obtain that $\lambda_v=1$ for all $v\in V_1$.

Now,  if $e=\{v,w\}\in E_1$ and $\tau(e)=\{\bar{v},\bar{w}\}\in E_2$, comparing
\begin{align*}
b_{\tau(e)}^2&=b_{\tau(e)}+b_{\bar{v}}+b_{\bar{w}},\text{ and}\\
\big(\lambda_e\psi(b_e)\big)^2&=\lambda_e^2\psi(b_e^2)=\lambda_e^2\psi(b_e+b_v+b_w)=\lambda_e^2\big(\psi(b_e)+\psi(b_v)+\psi(b_w)\big)\\
&=\lambda_eb_{\tau(e)}+\lambda_e^2\big(b_{\sigma(v)}+b_{\sigma(w)}\big),
\end{align*}
we get that $\lambda_e=1$ and $\tau(e)=\{\sigma(v),\sigma(w)\}$.

 In other words, the set morphism $\sigma\colon V_1\to V_2$, which has to be bijective since $\psi$ is an isomorphism, maps edges in $E_1$ to edges in $E_2$. That is $\sigma\colon\G_1\cong \G_2$ and moreover $\psi=\X(\sigma)$ as in Definition \ref{def:funtor_maps}.
\end{proof}

Finally, we check that $\X$ is full on the automorphism group of any object in the image. The proof follows similar arguments to those in the previous one:

\begin{lemma}\label{lem:aut-full}
Given a simple graph $\G=(V,E)$, and $g\in\Aut_{\E_\Bbbk}\big(\X(\G)\big)$, there exists  $f\in \Aut_\graphs(\G)$ such that $g=\X(f)$.
\end{lemma}
\begin{proof}
Let $g\in\Aut_{\E_\Bbbk}\big(\X(\G)\big)$. Recall that $\X(\G)$ is regular, and therefore $B$, the natural basis of $\X(\G)$ given in Definition \ref{def:funtor_objeto}, is unique up to permutation and scalar multiplication by Proposition \ref{prop:unique_basis}. Since  $g(B)$ is another natural basis of $\X(\G)$, and $g$ has to preserve the structure constants, then $g(b_v)=\lambda(v) b_{\sigma(v)}$, and $g(b_e)=\lambda(e) b_{\tau(e)}$ for some permutations $\sigma\in\text{Perm}(V)$, and $\tau\in\text{Perm}(E)$, and scalars $\lambda(v),\lambda(e)\in\Bbbk$.

Comparing
\begin{align*}
g(b_v^2)&=g(b_v)=\lambda(v)b_{\sigma(v)},\text{ and}\\
\big(g(b_v)\big)^2&=(\lambda(v)b_{\sigma(v)})^2=\lambda(v)^2b_{\sigma(v)},
\end{align*}
we obtain $\lambda(v)^2=\lambda(v)$ and therefore $\lambda(v)=1$.

If $e=\{v,w\}\in E$, and $\tau(e)=\{\bar{v},\bar{w}\}$, comparing
\begin{align*}
g(b_e^2)&=g(b_e+b_v+b_w)=\lambda(e)b_{\tau(e)}+\lambda(v)b_{\sigma(v)}+\lambda(w)b_{\sigma(w)} ,\text{ and}\\
\big(g(b_e)\big)^2&=(\lambda(e)b_{\tau(v)})^2=\lambda(e)^2\big(b_{\tau(e)}+b_{\bar{v}}+b_{\bar{w}}\big),
\end{align*}
we obtain $\lambda(e)^2=\lambda(e)$, thus $\lambda(e)=1$, and $\tau(e)=\{\sigma(v),\sigma(w)\}$. Therefore $\sigma$ is a permutation of the vertices of $\G$ that maps edges to edges. In other words, $\sigma\in \Aut_\graphs(\G)$ and $g=\X(\sigma)$ as in Definition \ref{def:funtor_maps}.
\end{proof}


\bibliographystyle{abbrv}
\bibliography{references}
\end{document}